\documentclass[10pt,reqno]{amsart}
\usepackage{amsmath}
\usepackage{amsfonts}
\usepackage{amssymb}
\usepackage{graphicx}
\usepackage{color}
\newtheorem{theorem}{Theorem}[section]\newtheorem{thm}[theorem]{Theorem}
\newtheorem*{theorem*}{Theorem}
\newtheorem{lemma}{Lemma}[section]
\newtheorem{corollary}[theorem]{Corollary}

\newtheorem{prop}{Proposition}[section]
\newtheorem{definition}[theorem]{Definition}

\newtheorem{remark}[theorem]{Remark}

\def\Ric{\text{Ric}}

\def\a{\alpha}
\def\l{\lambda}

\def\p{\partial}

\def\R{\mathbb{R}}

\def\vp{\varphi}

\def\Ric{\operatorname{Ric}}

\numberwithin{equation}{section}

\begin{document}

\title[First Robin Eigenvalue of the $p$-Laplacian on Riemannian Manifolds]{First Robin Eigenvalue of the $p$-Laplacian on Riemannian Manifolds}

\author{Xiaolong Li}
\address{Department of Mathematics, University of California, Irvine, Irvine, CA 92697, USA}
\email{xiaolol1@uci.edu}

\author{Kui Wang}\thanks{The research of the second author is supported by NSFC No.11601359} 
\address{School of Mathematical Sciences, Soochow University, Suzhou, 215006, China}
\email{kuiwang@suda.edu.cn}


\maketitle

\begin{abstract}  
We consider the first Robin eigenvalue $\l_p(M,\a)$ for the $p$-Laplacian on a compact Riemannian manifold $M$ with nonempty smooth boundary, with $\a \in \R$ being the Robin parameter. Firstly, we prove eigenvalue comparison theorems of Cheng type for $\l_p(M,\a)$. Secondly, when $\a>0$
we establish sharp lower bound of $\l_p(M,\a)$ in terms of dimension, inradius, Ricci curvature lower bound and boundary mean curvature lower bound, via comparison with an associated one-dimensional eigenvalue problem. The lower bound becomes an upper bound when $\a<0$.  
Our results cover corresponding comparison theorems for the first Dirichlet eigenvalue of the $p$-Laplacian when letting $\a \to +\infty$. 
\end{abstract}

\section{Introduction and Main Results}

The study of first nonzero eigenvalue for elliptic operators plays an important rule in both mathematics and physics, since this constant determines the convergence rate of numerical schemes in numerical analysis, describes the energy of a particle in the ground state in quantum mechanics, and determines the decay rate of heat flows in thermodynamics. Given its physical and mathematical significance, numerous bounds have been established for the first Dirichlet eigenvalue and the first nonzero Neumann eigenvalue of the Laplace operator (see for example \cite{Chavel84}\cite{ShapeOptmizationbook}\cite{LL10}\cite{SYbook}), and many results have been extended to the nonlinear $p$-Laplacian during the last two decades. 

The classical eigenvalue comparison theorem of Cheng \cite{Cheng75a} states that the first Dirichlet eigenvalue of a geodesic ball in an $n$-dimensional complete Riemannian manifold $M^n$ whose Ricci curvature is bounded from below by $(n-1)\kappa$, $\kappa \in \R$, is less than or equal to the first Dirichlet eigenvalue of a geodesic ball of the same radius in a space of constant sectional curvature $\kappa$, and the reversed inequality holds if we assume instead that the sectional curvature of $M$ is bounded from above by $\kappa$ and the radius of the geodesic ball is no greater than the injectivity radius at the center. 
For domains that are not geodesic balls or for general compact Riemannian manifolds with boundary, sharp lower bound estimates of the first Dirichlet eigenvalue of the Laplacian in terms of dimension $n$, inradius $R$, Ricci curvature lower bound $\kappa$, and boundary mean curvature lower bound $\Lambda$ were obtained by Li and Yau \cite{LY80} for $\kappa =\Lambda=0$ and by Kasue \cite{Kasue84} for general $\kappa, \Lambda \in \R$. 
The above-mentioned results have been generalized to the $p$-Laplacian for $1<p<\infty$. Matei \cite{Matei00} and Takeuchi \cite{Takeuchi98} proved Cheng's eigenvalue comparison theorems for the $p$-Laplacian, and Sakurai \cite{Sakurai19} obtained Li-Yau and Kause's theorem for the $p$-Laplacian on smooth metric measure spaces with boundary (including compact Riemannian manifolds with boundary). 

For either closed manifolds or compact manifolds with convex boundary and the Neumann boundary condition, sharp lower bound estimates of the first nonzero (closed or Neumann) eigenvalue in terms of dimension $n$, diameter $D$ and Ricci curvature lower bound $\kappa$ were established via the efforts of many mathematicians including Li \cite{Li79}, Li and Yau \cite{LY80}, Zhong and Yang \cite{ZY84}, Kr\"oger \cite{Kroger92}, and Bakry and Qian \cite{BQ00}. Their proofs use the gradient estimates method, together with comparisons with suitable one-dimensional models. Chen and Wang independently proved it using stochastic method \cite{CW94,CW95}.
In 2013, a simple alternative proof was given by Andrews and Clutterbuck \cite{AC13} (see also \cite{WZ17} for an alternative elliptic proof inspired by \cite{AC11,AC13} and \cite{Ni13}). They introduced the new method of estimating the modulus of continuity and the trick of reading the size of the first nonzero eigenvalue from the large time behavior of the diffusion. 
For the $p$-Laplacian, sharp lower bounds of the first nonzero eigenvalue, in terms of dimension, diameter and Ricci lower bound $\kappa$, were proved by Valtorta \cite{Valtorta12} for $\kappa=0$ and by Naber and Valtorta \cite{NV14} for general $\kappa \in \R$. Finally, extensions to the weighted Laplacian on Bakry-\'Emery manifolds were obtained by Bakry and Qian \cite{BQ00} via the gradient estimate method and by Andrews and Ni \cite{AN12} via the modulus of continuity approach, and further generalizations to the weighted $p$-Laplacian on smooth metric measure spaces were proved in \cite{Koerber18}\cite{LTW20}\cite{LW19eigenvalue, LW19eigenvalue2}\cite{Tu20}. It is also worth mentioning that the modulus of continuity estimates have been extended to viscosity solutions in \cite{Li16}\cite{LW17} and recently to fully nonlinear parabolic equations in \cite{Li20modulus}.

The Robin boundary condition $\frac{\p u}{\p \nu} +\a u =0$ ($\a \in \R$), interpolating the Neumann condition (with $\a =0$) and the Dirichlet condition (with $\a =+\infty$), however, did not receive as much attention as either of them.  
In thermodynamics, Robin boundary condition models heat diffusion with absorbing $(\a >0)$ or radiating $(\a<0)$ boundary. 
In recent years, many authors investigated eigenvalue problems with the Robin boundary condition (see for example \cite{ACH20}\cite{Laugesen19}\cite{LWW20}\cite{Savo20} and the references therein). In particular, Savo \cite{Savo20} established lower bound estimates of the first Robin eigenvalue of the Laplacian in terms of dimension, inradius, Ricci curvature lower bound and boundary mean curvature lower bound. When letting $\a \to \infty$, it reduces to the classical results of Li and Yau \cite{LY80} and Kasue \cite{Kasue84}.

The purpose of the present paper is to study the first Robin eigenvalue of the $p$-Laplacian on compact Riemannian manifolds with boundary. In particular, we will establish Cheng's eigenvalue comparison theorem (see Theorem \ref{Thm A} below) and sharp bounds for the first Robin eigenvalue of the $p$-Laplacian (see Theorem \ref{Thm C} below). 

Let $(M^n, g)$ be an $n$-dimensional smooth compact Riemannian manifold with smooth boundary $\p M \neq \emptyset$. 
Let $\Delta_p$ denote the $p$-Laplacian defined for $1<p<\infty$ by
\begin{equation*}
    \Delta_{p} u:= \text{div} (|\nabla u|^{p-2}\nabla u),
\end{equation*}
for $u\in W^{1,p}(M)$. When $p=2$, the $p$-Laplacian becomes the Laplacian. 
We consider the following eigenvalue problem with Robin boundary condition
\begin{equation}\label{eq 1.1}
    \begin{cases}
    -\Delta_{p} v  = \l |v|^{p-2} v, & \text{ in } M, \\
     \frac{\p v}{\p \nu}|\nabla v|^{p-2} +\a |v|^{p-2}v = 0, & \text{ on } \p M,  
    \end{cases}
\end{equation}
where $\nu$ denotes the outward unit normal vector field along $\p M$ and $\a \in \R$ is called the Robin parameter.  
The first Robin eigenvalue for $\Delta_p$, denoted by $\l_p(M,\a)$, is the smallest number  such that \eqref{eq 1.1} admits a weak solution in the distributional sense. Moreover, it can be characterized as 
\begin{equation}\label{eq 1.2}
    \l_{p}(M,\a) =\inf \left\{ \int_M |\nabla u|^p  d\mu_g +\a \int_{\p M} |u|^p dA: u \in W^{1,p}(M), \int_M |u|^p d\mu_g=1   \right\},
\end{equation}
where $d\mu_g$ is the Riemnnian measure induced by the metric $g$ and $dA$ is the induced measure on $\p M$. 
When $\a =0$, this reduces to the Neumann eigenvalue problem and we have $\l_{p}(M,0) =0$ with constants being the corresponding eigenfunctions. 
Hence we assume $\a \neq 0$ throughout the paper. 
It's easy to see from \eqref{eq 1.2} that  $\l_{p}(M,\a)>0$ if $\a >0$ and $\l_{p}(M,\a) < 0$ if $\a  < 0$. 
Indeed, $\l_{p}(M,\a)$ is an increasing function of $\a$ and it converges to the first Dirichlet eigenvalue of $\Delta_{p}$ as $\a \to +\infty$. 
Moreover, the first Robin eigenvalue $\l_p(M,\a)$ 
is simple and the first eigenfunction has a constant sign, thus can always be chosen to be positive. Note that the first eigenfunction is in general not smooth if $p\neq 2$, but belongs to $C^{1,\gamma}(\overline{M})$ for some $0<\gamma<1$, as proved by L\^{e} \cite{Le06}. 

We now state Cheng's eigenvalue comparison theorem for $\l_p(M,\a)$, 
which seems to be new even for the Laplacian. 
\begin{thm}\label{Thm A}
Let $M^n(\kappa)$ denote the simply-connected $n$-dimensional space form with constant sectional curvature $\kappa$ and let $V(\kappa,R)$ be a geodesic ball of radius $R$ in $M^n(\kappa)$. 
Let $M^n$ be an $n$-dimensional complete Riemannian manifold and $B_R(x_0) \subset M $ be the geodesic ball of radius $R$ centered at $x_0$.
(We always have $R < \frac{\pi}{\sqrt{\kappa}}$ if $\kappa >0$ in view of the Myers theorem).
\begin{enumerate}
    \item Suppose $\Ric \geq (n-1)\kappa$ on $B_R(x_0)$. Then
    \begin{align*}
        \l_p(B_R(x_0), \a) & \le \l_p(V(\kappa, R), \a), \text{ if } \a >0, \\
        \l_p(B_R(x_0), \a) & \ge \l_p(V(\kappa, R), \a),
        \text{ if } \a < 0.
    \end{align*}
    \item Let $\Omega \subset B_R(x_0)$ be a domain with smooth boundary. 
    Suppose $\operatorname{Sect} \leq \kappa$ on $\Omega$ and $R$ is less than the injectivity radius at $x_0$. Then 
    \begin{align*}
        \l_p(\Omega, \a) & \ge \l_p(V(\kappa, R), \a), \text{ if } \a >0, \\
        \l_p(\Omega, \a) & \le \l_p(V(\kappa, R), \a),
        \text{ if } \a < 0.
    \end{align*}
\end{enumerate}
Moreover, the equality holds
    if and only if $B_R(x_0)$ (or $\Omega$) is isometric to $V(\kappa, R)$.
\end{thm}
\begin{remark}
Letting $\a \to +\infty$ in Theorem \ref{Thm A} yields Cheng's eigenvalue comparison theorems proved in \cite{Cheng75a} for the Laplacian and \cite{Matei00}\cite{Takeuchi98} for the $p$-Laplacian.
\end{remark}
\begin{remark}
Domain monotonicity (if $\Omega_1 \subset \Omega_2$, then the first Dirichlet eigenvalue of $\Omega_1$ is bigger than that of $\Omega_2$) is a fundamental property for first Dirichlet eigenvalue, but it fails for the first Robin eigenvalue, even for convex Euclidean domains \cite{GS05}. Part (2) of Theorem \ref{Thm A} can be viewed as a domain monotonicity result for $\l_p(M,\a)$, as it implies that domain monotonicity holds for $\a>0$ (reversed domain monotonicty for $\a <0$) in space forms when the outer domain is a ball. Indeed, our proof shows that when the outer domain is a ball, domain monotonicty holds on the warped product manifolds of the form $[0,T] \times S^{n-1}$ with metric $g=dr^2 + f^2(r) g_{S^{n-1}}$, provided that the warping function $f$ is strictly log-concave. 
\end{remark}

We introduce some notations before stating the next theorem. Let $R$ denote the inradius of $M$ defined by 
$$R=\sup\{d(x, \p M):x\in M \}.$$ 
Let $C_{\kappa, \Lambda}(t)$ be the unique solution of 
\begin{equation*}
\begin{cases}
C_{\kappa, \Lambda}'' +\kappa \, C_{\kappa, \Lambda}(t) =0, \\ C_{\kappa, \Lambda}(0)=1, \\
C_{\kappa, \Lambda}'(0) =-\Lambda,
\end{cases}
\end{equation*}
and define
$$T_{\kappa,\Lambda}(t) :=\frac{C'_{\kappa, \Lambda}(t)}{C_{\kappa, \Lambda}(t)}.$$ 

Our second main theorem states 
\begin{thm}\label{Thm C}
Let $(M^n,g)$ be a compact Riemannian manifold with boundary $\p M \neq \emptyset$. 
Suppose that the Ricci curvature of $M$ is bounded from below by $(n-1)\kappa$ and the mean curvature of $\p M$ is bounded from below by $(n-1)\Lambda$ for some $\kappa, \Lambda \in \R$. 
Let $\l_{p}(M,\a)$ be the first Robin eigenvalue of the $p$-Laplacian on $M$. Then 
\begin{align*}
   & \l_{p}(M,\a) \geq \bar\l_p\left([0,R],\a \right), \text{ if } \a >0, \\
   & \l_{p}(M,\a) \leq \bar\l_p\left([0,R],\a \right), \text{ if } \a <0,
\end{align*}
where $\bar\l_p\left([0,R],\a \right)$
is the first eigenvalue of the one-dimensional eigenvalue problem 
\begin{equation}\label{eq 1.4}
    \begin{cases}
   (p-1)|\vp'|^{p-2} \vp'' +(n-1)T_{\kappa, \Lambda} |\vp'|^{p-2}\vp' =-\l |\vp|^{p-2}\vp, \\
     |\vp'(0)|^{p-2}\vp'(0)=\a |\vp(0)|^{p-2}\vp(0), \\
    \vp'(R)=0.
    \end{cases}
\end{equation}
%
Moreover, the equality occurs if and only if $(M,g)$ is a $(\kappa, \Lambda)$-model space defined in Definition \ref{def model space}. 
\end{thm}

\begin{remark}
When $p=2$, the above theorem is due to Savo \cite{Savo20}. His proof made use of the Green's formula and does not seem to work for the $p$-Laplacian. Our proof uses instead a Picone's identity for $\Delta_p$ proved in \cite{AH98}. 
\end{remark}

\begin{remark}
Letting $\a \to +\infty$ in Theorem \ref{Thm C} yields the optimal lower bound for the first Dirichlet eigenvalue of $\Delta_{p}$, which was obtained by Kasue \cite{Kasue84} for $p=2$, and by Sakurai \cite{Sakurai19} for general $1< p <\infty$. 
\end{remark}

\begin{remark}
As in \cite[pages 26-28]{Savo20}, given any $\kappa, \Lambda \in \R$ and $R>0$, one can construct an $n$-dimensional manifold $\bar\Omega:=\bar\Omega(\kappa, \Lambda, R)$ with two boundary component $\Gamma_1$ and $\Gamma_2$, such that the first eigenvalue of $\Delta_p$ on $\bar\Omega$ with Robin boundary condition on $\Gamma_1$ and Neuman boundary condition on $\Gamma_2$ coincides with $\bar\l_p([0,R],\a)$.  
\end{remark}

The boundary conditions of the one-dimensional eigenvalue problem \eqref{eq 1.4} are Robin at $t=0$ and Neumann at $t=R$. When $\kappa =\Lambda =0$, the first eigenvalue of problem \eqref{eq 1.4} (with $T_{\kappa, \Lambda} \equiv 0$) is indeed equal to the first Robin eigenvalue of problem \eqref{eq 1.6} (see Proposition \ref{prop 2.2} below). Thus, we get an eigenvalue comparison theorem between the first Robin eigenvalue of the $n$-dimensional manifold $M$ and the first Robin eigenvalue of a one-dimensional eigenvalue problem.  
\begin{thm}
Let $(M^n,g)$ be the same as in Theorem \ref{Thm C}. Suppose $\kappa =\Lambda=0$. Then 
\begin{align*}
    &\l_p(M,\a) \geq \mu_p\left([0,2R], \a \right) \text{ if } \a >0, \\
    & \l_p(M,\a) \leq \mu_p\left([0,2R], \a \right) \text{ if } \a < 0, 
\end{align*}
where $\mu_p\left([0,2R], \a \right) $ is the first Robin eigenvalue of the one-dimensional problem 
\begin{equation}\label{eq 1.6}
    \begin{cases}
   (p-1)|\vp'|^{p-2} \vp''  =- \l |\vp|^{p-2}\vp, \\
   |\vp'(0)|^{p-2}\vp'(0)=\a |\vp(0)|^{p-2}\vp(0), \\
    |\vp'(2R)|^{p-2}\vp'(2R)=-\a |\vp(2R)|^{p-2}\vp(2R).
    \end{cases}
\end{equation}
\end{thm}

The paper is organized as follows. 
In section 2, we collect some basic properties of the one-dimensional eigenvalue problems. 
An extension of Barta's inequality for the $p$-Laplacian is given in section 3.
The proofs of Theorem \ref{Thm A} and Theorem \ref{Thm C} are presented in section 4 and section 5, respectively. 
The model spaces on which the inequalities in Theorem \ref{Thm C} are achieved are provided in section 6. 

\section{Properties of one-dimensional Models}
In this section, we gather several basic properties of the one-dimensional eigenvalue problems used as comparison models in this paper. 

We will consider slightly more general models. Let $w$ be a positive smooth function on $[0,R]$ satisfying $w(0)=1$. We consider the following one-dimensional eigenvalue problem:
\begin{equation}\label{eq 2.1}
     \begin{cases}
     (p-1)|\vp'|^{p-2} \vp'' + \frac{w'}{w}|\vp'|^{p-2}\vp' =-\l |\vp|^{p-2}\vp, \\
    |\vp'(0)|^{p-2}\vp'(0)=\a |\vp(0)|^{p-2}\vp(0), \\
    \vp'(R)=0.
    \end{cases}
\end{equation}
Let $\bar\l_p([0,R], w, \a)$ be the first eigenvalue of \eqref{eq 2.1}. It's easily seen that
 $\bar\l_p([0,R], w, \a)$ is characterized by
\begin{eqnarray}\label{eq 2.2}
    && \bar\l_p([0,R], w, \a) \nonumber \\
    &=&\inf \left\{\int_0^R |u'|^p w dt +\a  |u(0)|^p : u \in W^{1,p}\left([0,R], w dt\right), \int_0^R |u|^p w dt =1 \right\} .
\end{eqnarray}
It follows from \eqref{eq 2.2} that $\bar\l_p([0,R], w, \a) =0$ if $\a =0$, $\bar\l_p([0,R], w, \a) > 0$ if $\a >0$, and $\bar\l_p([0,R], w, \a) <0$ if $\a <0$. 
Moreover, the first eigenfunction does not change sign and can always be chosen to be positive. 

We prove the following properties of the first eigenfunction:
\begin{prop}\label{prop 2.1}
Let $u>0$ be the positive first eigenfunction associated to $\bar\l_p ([0,R], w, \a)$.
\begin{itemize}
    \item[(1)] If $\a>0$, then $u'>0$ on $[0,R)$.
    \item[(2)] If $\a<0$, then $u'<0$ on $[0,R)$.
    \item[(3)] If $\a >0$ and $\bar{R} < R$, then $\bar\l_p([0,R],w,\a) < \bar\l_p([0,\bar{R}], w, \a) $.
\end{itemize}
  In (4) and (5), assume further that $w$ is strictly log-concave, i.e., $(\log w )'' <0$ on $[0,R)$. \begin{itemize}
    \item[(4)] If $\a >0$, then $\frac {u'}{u}$ is monotone decreasing on $[0,R]$. Particularly, $|\frac{u'}{u}|^{p-1} \leq \a $ on $[0,R]$.
    \item[(5)] If $\a <0$, then $\frac{u'}{u}$ is monotone increasing on $[0,R]$. Particularly, $|\frac{u'}{u}|^{p-1} \leq -\a $ on $[0,R]$.
\end{itemize}
\end{prop}
\begin{proof}
(1). If $\a>0$, then $u'(0)>0$ because $|u'(0)|^{p-2} u'(0) =\a |u(0)|^{p-2} u(0) >0$. 
We argue by contradiction and let $r \in (0,R)$ be the first zero of $u'$. Define $v \in W^{1,p}([0,R], w dt)$ by
\begin{equation*}
    v(t) =
    \begin{cases}
    u(t),  & \text{ for } 0\leq t \leq r, \\
    u(r),  & \text{ for } r \leq t\leq R. 
    \end{cases}
\end{equation*}
Then using $v$ as a test function in \eqref{eq 2.2} gives 
\begin{eqnarray*}
&& \int_0^R |v'|^p w\, dt +\a |v(0)|^p \\
&=&\int_0^{r} |u'|^p w\, dt +\a |u(0)|^p \\
&=&u|u'|^{p-2}u'w|_0^{r}-\int_0^{r} (|u'|^{p-2}u' w)'u\, dt +\a w(0)|u(0)|^p\\
&=&\bar\l_p([0,R],w, \a)\int_0^{r} |u|^p w\, dt\\ &<&\bar\l_p([0,R],w,\a)\int_0^{R} |v|^p w\,  dt,
\end{eqnarray*}
contradicting \eqref{eq 2.2}. Thus we have $u'>0$ on $[0, R)$. 

(2). Argue by contradiction again. Let $r\in (0,R)$ be the first zero of $u'$. Then $u$ restricted to $[r,R]$ is an Neumann eigenfunction with Neumann eigenvalue  $\bar\l_p([0,R], w, \a)$ on $[r, R]$. 
This is impossible since Neumann eigenvalue are nonnegative while $\bar\l_p([0,R], w, \a) <0$ if $\a <0$. 

(3). Similar to (1), one prolongs an eigenfunction of $[0,\bar{R}]$ on $[0,R]$ by a constant and use it as a test function in \eqref{eq 2.2} to derive a contradiction. 

(4). Let $v(t)=\frac{u'(t)}{u(t)}$, then $v(0)=\a^{\frac{1}{p-1}}$, $v(r)>0$ for $r\in[0,R)$ and $v(R)=0$.
Direct calculation using 
$$
(p-1)|u'|^{p-2} u'' + \frac{w'}{w}|u'|^{p-2}u' =-\bar\l_p([0,R],w,\a) |u|^{p-2}u
$$
yields 
\begin{equation}\label{eq 2.3}
(v|v|^{p-2})'+\frac{w'}{w}|v|^{p-2}v+(p-1)|v|^p=-\bar\l_p([0,R],w,\a).
\end{equation}
Now we claim that $v(r)$ is monotone decreasing on $[0, R]$. If not, there exists some $r\in(0,R)$ such that
$$v'(r)=0, v''(r)\le 0.$$
Then taking derivative of \eqref{eq 2.3}, we have at $t=r$ that
$$
0=(p-1) v''|v|^{p-2}+\left(\frac{w'}{w}\right)'|v|^{p-2}v<0
$$
by the strict log-concavity of $w$, which is clearly a contradiction.

(5). Similar to the proof of (4). 
\end{proof}

\begin{prop}\label{prop 2.2}
Let $\mu_p([0,2R], \a)$ be the first eigenvalue of the following eigenvalue problem:
\begin{equation}\label{eq 2.4}
     \begin{cases}
      (p-1)|\vp'|^{p-2} \vp'' =-\l |\vp|^{p-2}\vp, \\
    |\vp'(0)|^{p-2}\vp'(0)=\a |\vp(0)|^{p-2}\vp(0), \\
    |\vp'(2R)|^{p-2}\vp'(2R)=-\a |\vp(2R)|^{p-2}\vp(2R) .
    \end{cases}
\end{equation}
Then
\begin{equation*}
    \mu_p([0,2R], \a)= \bar\l_p([0,R], 1, \a).
\end{equation*}
\end{prop}
\begin{proof}
Observe that \eqref{eq 2.4} is invariant under the symmetry $t \to 2R-t$. It then follows that, if we fix a positive first eigenfunction $v$ of \eqref{eq 2.4}, then $v$ must be even at $t=R$ ($v$ cannot be odd at $t=R$ since $v$ is positive). Hence $v'(R)=0$ and $v$ is also eigenfunction of \eqref{eq 2.1} (with $w \equiv 1$). It has to be the first eigenfunction since $v$ is positive. 

Conversely, the first eigenfunction $u$ of \eqref{eq 2.1} can be extended to a function $\bar{u}$ on $[0,2R]$ by
\begin{equation*}
    \bar{u} (t) = 
    \begin{cases} 
    u(t), & \text { for } 1\leq t \leq R, \\
    u(2R -t), & \text{ for } R \leq t \leq 2R. 
    \end{cases}
\end{equation*}
It's easy to see that $\bar{u}$ is the first eigenfunction of \eqref{eq 2.4}.

\end{proof}

\section{An Extension of Barta's Inequality}
Barta's inequality (see for example \cite[Lemma 1 on page 70]{Chavel84}) was frequently used in obtaining lower and upper for the first Dirichlet eigenvalue of the Laplacian. 
It asserts that for any function $v\in C^2(M) \cap C(\overline{M})$ satisfying $v>0$ in $M$ and $v=0$ on $\p M$, 
\begin{equation*}
   \inf_M \frac{-\Delta v}{v} \leq  \l_1^D (M)  \leq \sup_M \frac{-\Delta v}{v},
\end{equation*}
where $\l_1^D(M)$ is the first Dirichlet eigenvalue of the Laplacian on $M$. 
Kasue \cite[Lemma 1.1]{Kasue84} extended Barta's result and proved that if there is a positive continuous function $v$ on $M$ satisfying $-\Delta v  \geq \l v $ in the distributional sense for some constant $\l$, then $\l_1^D(M) \geq \l$.
Moreover, if $v$ is smooth on an open dense subset of $M$ and the equality is achieved, then $v$ is the first eigenfunction satisfying the Dirichlet boundary condition.
We extend these results to the $p$-Laplacian with Robin boundary condition. 

\begin{thm}\label{Thm Barta}
Let $v\in C^1(\overline{M})$ (it suffices to assume $v$ is Lipschitz on $\overline{M}$) be a positive function. 
\begin{enumerate}
    \item Suppose that $v$ satisfies 
    \begin{equation*}
    \begin{cases}
    -\Delta_{p} v  \geq \l |v|^{p-2} v, & \text{ in } M, \\
    \frac{\p v}{\p \nu} |\nabla v|^{p-2} +\a |v|^{p-2}v \geq 0, & \text{ on } \p M,  
    \end{cases}
\end{equation*}
in the distributional sense, then we have 
\begin{equation*}
    \l_{p}(M, \a) \geq \l.
\end{equation*}
Moreover, the equality holds if and only if $v$ is a constant multiple of the first eigenfunction of $\l_p(M,\a)$. 
\item Suppose that $v$ satisfies 
\begin{equation*}
    \begin{cases}
    -\Delta_{p} v  \leq \l |v|^{p-2} v, & \text{ in } M, \\
    \frac{\p v}{\p \nu} |\nabla v|^{p-2} +\a |v|^{p-2}v \leq 0, & \text{ on } \p M,  
    \end{cases}
\end{equation*}
in the distributional sense, then we have 
\begin{equation*}
    \l_{p}(M, \a) \leq \l.
\end{equation*}
\end{enumerate}
\end{thm}

An immediate consequence of the rigidity in part (1) of Theorem \ref{Thm Barta} is the simpleness of $\l_p(M,\a)$. 
\begin{corollary}
$\l_p(M,\a)$ is simple. 
\end{corollary}

To prove Theorem \ref{Thm Barta}, we make use of the Picone's identity for $\Delta_{p}$ proved in \cite{AH98}. For reader's convenience, we include its short proof here as well. 
\begin{prop}[Picone's identity]\label{prop Picone}
Let $u\geq 0$ and $v>0$ be differentiable functions on $M$. Let  
\begin{align*}
    L(u,v) &=|\nabla u|^p +(p-1)\frac{u^p}{v^p}|\nabla v|^p -p\frac{u^{p-1}}{v^{p-1}} |\nabla v|^{p-2} \langle \nabla u, \nabla v \rangle, \\
    R(u,v) &=|\nabla u|^p - |\nabla v|^{p-2} \left\langle \nabla \left(\frac{u^p}{v^{p-1}} \right), \nabla v \right\rangle.
\end{align*}
Then $$L(u,v)=R(u,v) \geq 0.$$ Moreover, $L(u,v) =0$ a.e. in $M$ if and only if $u=c v$ for some constant $c$.
\end{prop}

\begin{proof}
Direct calculation gives
\begin{eqnarray*}
 R(u,v) &=&|\nabla u|^p - |\nabla v|^{p-2} \left\langle \nabla \left(\frac{u^p}{v^{p-1}} \right), \nabla v \right\rangle\\
 &=&|\nabla u|^p +(p-1)\frac{u^p}{v^p}|\nabla v|^p -p\frac{u^{p-1}}{v^{p-1}} |\nabla v|^{p-2} \langle \nabla u, \nabla v \rangle\\
 &=&L(u,v).
\end{eqnarray*}
Applying H\"older's inequality $ab\le\frac{a^p}{p}+\frac{p-1}{p}b^{\frac{p}{p-1}}$ with $a=|\nabla u|$ and $b=\frac{u^{p-1}}{v^{p-1}}|\nabla v|^{p-1}$, we have
\begin{eqnarray*}
\frac{u^{p-1}}{v^{p-1}} |\nabla v|^{p-2} |\langle \nabla u, \nabla v \rangle|&\le&\frac{u^{p-1}}{v^{p-1}}|\nabla v|^{p-1}|\nabla u|\\
&\le&\frac{|\nabla u|^p}{p}+\frac{p-1}{p}\frac{u^p}{v^p}|\nabla v|^p,
\end{eqnarray*}
proving that $L(u,v)\ge 0$. 
If the equality occurs, then we easlily conclude that $\nabla \left(\frac{u}{v} \right) =0$ a.e. on $M$ and consequently $u= c\, v$ for some constant $c$. 
\end{proof}

\begin{proof}[Proof of Theorem \ref{Thm Barta}]
(1). By assumption, we have  
\begin{equation*}
    \int_M |\nabla v|^{p-2} \langle \nabla v, \nabla \eta \rangle \, d\mu_g \geq \l \int_M v^{p-1} \eta \, d\mu_g -\a \int_{\p M} v^{p-1} \eta \, dA,  
\end{equation*}
for any nonnegative function $\eta \in C^1(\overline{M})$. Choosing $\eta =\frac{\vp^p}{v^{p-1}}$ for any smooth function $\vp$ gives 
\begin{eqnarray*}
 \int_M |\nabla v|^{p-2} \left\langle \nabla v, \nabla \left(\frac{\vp^p}{v^{p-1}}\right) \right\rangle \, d\mu_g \geq \l \int_M \vp^p \, d\mu_g -\a \int_{\p M} \vp^p \, dA,  
\end{eqnarray*}
On the other hand, Picone's identity in Proposition \ref{prop Picone} implies 
\begin{equation*}
    \int_M |\nabla v|^{p-2} \left\langle \nabla v, \nabla \left(\frac{\vp^p}{v^{p-1}}\right) \right\rangle \, d\mu_g \leq \int_M |\nabla \vp |^p \, d\mu_g .
\end{equation*}
Combing the above two inequalities together yields 
\begin{equation*}
    \int_M |\nabla \vp |^p \,  d\mu_g  + \a \int_{\p M} \vp^p \, dA \geq  \l \int_M \vp^p \, d\mu_g .
\end{equation*}
The desired inequality $\l_p(M,\a) \geq \l$ follows by letting $\vp$ approach the first eigenfunction $u \in W^{1,p}(M)$. 
The equality occurs only if $\int_M L(\vp,v)  d\mu_g =0 $, which implies $L(u,v)=0$ a.e. on $M$, and then the equality case in Proposition \ref{prop Picone} implies $u=c\, v$ for some constant $c$. \\
(2). By assumption, we have that for any nonnegative function
$\eta  \in C^1(\overline{M})$,
\begin{equation*}
    \int_M |\nabla v|^{p-2} \langle \nabla v, \nabla \eta \rangle \, d\mu_g \le \l \int_M v^{p-1} \eta \, d\mu_g -\a \int_{\p M} v^{p-1} \eta \, dA.
\end{equation*}
Letting $\eta$ approach $v$ yields
\begin{equation*}
   \int_M |\nabla v|^p \, d\mu_g \le \l \int_M v^{p}  \, d\mu_g -\a \int_{\p M} v^{p}  \, dA,
\end{equation*}
which implies $\l_p(M,\a)  \leq \l$. 
\end{proof}

\section{Proof of Theorem \ref{Thm A}}
In this section, we prove Cheng's eigenvalue comparison theorems for $\l_p(M,\a)$. 
By Barta's inequality, we need to construct sub and supersolution for the eigenvalue equation. 
A natural function on $M$ is the distance function from a given point. Let $p\in M$ and $r(x)=d(x,p)$ be the geodesic distance from $p$ to $x$. Then $r(x)$ is a Lipschitz continuous function on $M$ and it is smooth on $M\setminus\{p, \text{Cut}(p)\}$, where $\text{Cut}(p)$ denotes the cut locus of $p$. 

Let $sn_\kappa$ be the unique solution of $sn_\kappa''+\kappa sn_\kappa=0 $ with $sn_\kappa (0)=0$ and $sn_\kappa'(0)=1$, i.e.,  $sn_\kappa$ are the coefficients of the Jacobi fields of the model spaces $M^n(\kappa)$ given by 
\begin{equation*}
    sn_\kappa(t) = 
    \begin{cases}
    \frac{1}{\sqrt{\kappa}} \sin(\sqrt{\kappa} t) & \text{ if } \kappa >0,\\
    t & \text{ if } \kappa=0,\\
    \frac{1}{\sqrt{-\kappa}}\sinh{(\sqrt{-\kappa}t)} & \text{ if } \kappa<0.
    \end{cases}
\end{equation*}

We need the Laplace comparison theorem for the distance function, see for example \cite{SYbook} and \cite{Sakai96}. 
\begin{thm}\label{thm comparison}
Let $(M^n,g)$ be a complete Riemannian manifold of dimension $n$ and $p\in M$. Let $r(x)=d(x,p)$.  
\begin{enumerate}
    \item Suppose that $\Ric \geq (n-1)\kappa$ on $M$. Then 
    \begin{equation*}
        \Delta r(x) \leq (n-1)\frac{sn_\kappa'(r)}{sn_\kappa(r)}
    \end{equation*}
    holds for all $x \in M\setminus\{p, \emph{\text{Cut}}(p) \}$, and also holds globally on $M$ in the sense of distribution. 
 
    \item Suppose that $\operatorname{Sect} \leq \kappa$ on M.  Then 
    \begin{equation*}
        \Delta r(x) \geq (n-1)\frac{sn_\kappa'(r)}{sn_\kappa(r)}
    \end{equation*}
   holds on the set $\{x \in M: r(x) \leq \min\{\emph{\text{inj}}(p), \frac{\pi}{2\sqrt{\kappa}}\}\}$, 
   where $\emph{\text{inj}}(p)$ denotes the injectivity radius at $p$, and we understand $\frac{\pi}{2\sqrt{\kappa}} =\infty$ if $\kappa \leq 0$. 
\end{enumerate}
\end{thm}
\begin{proof}[Proof of Theorem \ref{Thm A}]
(1). 
We first deal with the $\a>0$ case. 
Let $u$ be the first positive eigenfunction associated with $\l_p(V(\kappa, R), \a)$, which is a radial function, given by $u(x)=\vp(r(x))$, where $r(x)=d(x,x_0)$ and $\vp$ satisfies 
\begin{equation*}
    \begin{cases}
    (p-1)|\vp'|^{p-2} \vp'' +(n-1)\frac{sn'_\kappa}{sn_\kappa} |\vp'|^{p-2}\vp' =-\l_p(V(\kappa,R),\a) |\vp|^{p-2}\vp, \\
     |\vp'(R)|^{p-2}\vp'(R)=-\a |\vp(R)|^{p-2}\vp(R), \\
    \vp'(0)=0.
    \end{cases}
\end{equation*}
Applying Proposition \ref{prop 2.1} with $w=sn_\kappa^{n-1}$ and $u(t)=\vp(R-t)$, we have that 
$\vp'(t)<0$. 
Consider the function $v(x)$ defined on $B_R(x_0)$ by
$$
v(x)=\vp(r(x)).
$$
Since $\Ric \ge (n-1)\kappa$, we have $\Delta r(x) \le (n-1)\frac{sn_\kappa'(r)}{sn_\kappa(r)}$ for all $x \in M\setminus\{x_0, C(x_0) \}$ by part (1) of Theorem \ref{thm comparison}. 
Direct calculation gives  
\begin{eqnarray*}
-\Delta_p v (x)&=&-(p-1)|\vp'|^{p-2}\vp''-|\vp'|^{p-2}\vp'\Delta r(x)\\
&\le&-(p-1)|\vp'|^{p-2}\vp''-|\vp'|^{p-2}\vp'\frac{(n-1)sn'_\kappa}{sn_\kappa}\\
&=&\l_p(V(\kappa, R),\a) |\vp|^{p-2}\vp\\
&=&\l_p(V(\kappa, R),\a) |v|^{p-2}v
\end{eqnarray*}
for all $x \in M\setminus\{x_0, C(x_0) \}$.
On the other hand, direct calculation shows that $v(x)$ satisfies $$|\nabla v|^{p-2} \frac{\p v}{\p \nu} +\a |v|^{p-2} v =0$$ on $\p B_R(x_0)$.
Since the cut locus is a null set, standard argument via approximation shows that $v(x)$ satisfies 
\begin{equation*}
    \begin{cases}
    -\Delta_{p} v  \leq \l_p(V(\kappa, R),\a) |v|^{p-2} v, & \text{ in } B_{R}(x_0), \\
    \frac{\p v}{\p \nu} |\nabla v|^{p-2} +\a |v|^{p-2}v = 0, & \text{ on } \p  B_{R}(x_0),  
    \end{cases}
\end{equation*}
in the distributional sense. 
It then follows from part (2) of Theorem \ref{Thm Barta} that 
\begin{equation*}
     \l_p(B_{R}(x_0), \a)\le \l_p(V(\kappa, R), \a).
\end{equation*}

If $\a <0$, we have that 
$\vp'(t)>0$. Same argument as in the $\a>0$ case shows that 
 $v(x)$ satisfies 
\begin{equation*}
    \begin{cases}
    -\Delta_{p} v  \geq \l_p(V(\kappa, R),\a) |v|^{p-2} v, & \text{ in } B_R(x_0), \\
    \frac{\p v}{\p \nu} |\nabla v|^{p-2} +\a |v|^{p-2}v = 0, & \text{ on } \p B_R(x_0),  
    \end{cases}
\end{equation*}
in the distributional sense. 
The desires estimate $\l_p(B_R(x_0), \a)  \ge \l_p(V(\kappa, R), \a) $ follows from part (1) of Theorem \ref{Thm Barta}. 

(2). If $\a>0$, then we have from Proposition \ref{prop 2.1} that
$$\vp'(t)<0 \text{ and } \frac{\vp'(t)}{\vp(t)}\ge -\a^{\frac{1}{p-1}}$$
for $t\in(0,R]$. 
Since $\operatorname{Sect}\le \kappa$, we have $\Delta r(x) \ge (n-1)\frac{sn_\kappa'(r)}{sn_\kappa(r)}$ by part (2) of Theorem \ref{thm comparison}. 
Firstly, same argument as in the proof of (1) shows that 
$$-\Delta_p v \geq \l_p(V(\kappa, R),\a) |v|^{p-2}v $$ on $\Omega$.
Secondly, using $\frac{\p v}{\p \nu_{\Omega}} =\vp' \langle \nabla r, \nu_\Omega \rangle \ge \vp'$ on $\p \Omega$, we estimate that
\begin{equation*}
    |\nabla v|^{p-2} \frac{\p v}{\p \nu_{\Omega}} +\a |v|^{p-2}v \ge |\vp'|^{p-2} \vp' +\a |\vp|^{p-2}\vp \ge 0
\end{equation*}
on $\p \Omega$, where $\nu_\Omega$ denote the unit outward normal vector field along $\p \Omega$.
Thus we conclude
\begin{equation*}
    \begin{cases}
    -\Delta_{p} v  \ge \l_p(V(\kappa, R),\a) |v|^{p-2} v, & \text{ in } \Omega, \\
    \frac{\p v}{\p \nu} |\nabla v|^{p-2} +\a |v|^{p-2}v \ge 0, & \text{ on } \p \Omega,  
    \end{cases}
\end{equation*}
holds in the distributional sense. 
The desires estimate $\l_p(\Omega, \a)  \ge \l_p(V(\kappa, R), \a) $ follows from part (1) of Theorem \ref{Thm Barta}. 

\end{proof}

\section{Proof of Theorem \ref{Thm C}}

By Barta's inequality in Theorem \ref{Thm Barta}, we need to find sub and supersolution to the eigenvalue equation for $\Delta_p$ with Robin boundary condition, in order to establish lower and upper bounds for $\l_p(M,\a)$. 
The natural choice here is the distance function to the boundary $d(x, \p M)$.
It is well known that the function $d(x, \p M)$ is Lipschitz on $M$ and smooth on $M\setminus \text{Cut}(\p M)$, where $\text{Cut}(\p M)$ denotes the cut locus of $\p M$ and it is a null set. 
We recall the following Laplace comparison theorem for $d(x, \p M)$ (see for instance \cite{Kasue82}). 
\begin{thm}\label{thm laplace comparison}
Let $(M^n,g)$ be a compact Riemannian manifold with boundary $\p M \neq \emptyset$. 
Suppose that the Ricci curvature of $M$ is bounded from below by $(n-1)\kappa$ and the mean curvature of $\p M$ is bounded from below by $(n-1)\Lambda$ for some $\kappa, \Lambda \in \R$. Then 
\begin{equation*}
    \Delta d(x, \p M) \leq  (n-1)T_{\kappa, \Lambda}\left(d(x,\p M)\right),
\end{equation*}
on $M\setminus \emph{\text{Cut}}(\p M)$.
\end{thm}

We then construct sub and supersolution of the eigenvalue equation by composing $d(x, \p M)$ with the eigenfunction of the one-dimensional problem \eqref{eq 1.4}. 
\begin{prop}\label{prop eigenfunction}
Let $\bar\l_p :=\bar\l_p\left([0,R],\a \right)$ and $\vp$ be the first eigenvalue and eigenfunction of the one-dimensional problem \eqref{eq 1.4}.
Let $v(x)=\vp(d(x,\p M))$. 
\begin{enumerate}
    \item If $\a >0$, then $v$ satisfies 
\begin{equation*}
    \begin{cases}
    -\Delta_{p} v  \geq \bar\l_p |v|^{p-2} v, & \text{ in } M, \\
    \frac{\p v}{\p \nu} |\nabla v|^{p-2} +\a |v|^{p-2}v \geq 0, & \text{ on } \p M,  
    \end{cases}
\end{equation*}
in the distributional sense. 
\item If $\a < 0$, then $v$ satisfies 
\begin{equation*}
    \begin{cases}
    -\Delta_{p} v  \leq \bar\l_p |v|^{p-2} v, & \text{ in } M, \\
    \frac{\p v}{\p \nu} |\nabla v|^{p-2} +\a |v|^{p-2}v \leq 0, & \text{ on } \p M,  
    \end{cases}
\end{equation*}
in the distributional sense. 
\end{enumerate}
\end{prop}
It's easy to see that $v(x)$ satisfies the Robin boundary condition $\frac{\p v}{\p \nu}|\nabla v|^{p-2} +\a |v|^{p-2}v =0$ on $\p M$ and the inequality $-\Delta_p v \geq \bar\l_p |v|^{p-2} v$ holds on $M\setminus \text{Cut}(M)$ if $\a >0$. To show the partial differential inequality holds in the sense of distribution, we need the following lemma in \cite[Lemma 2.5]{Sakurai19}, which is useful in avoiding the cut locus of $\p M$. 
\begin{lemma}
Let $(M,g)$ be a smooth Riemannian manifold with smooth boundary $\p M$. Then there exists a sequence $\{\Omega_k\}_{k=1}^\infty$ of closed subsets of $\overline{M}$ satisfying the following properties:
\begin{enumerate}
    \item for every $k$, the set $\p \Omega_k$ is a smooth hypersurface in $M$ and $\p \Omega_k \cap \p M =\p M$;
    \item $k_1 < k_2$ implies $\Omega_{k_1} \subset \Omega_{k_2}$;
    \item $\overline{M} \setminus \emph{\text{Cut}}(M) =\cup_{k=1}^\infty \Omega_k$; 
    \item for every $k$, on $\p \Omega_k \setminus \p M$, there exists the unit outward normal vector field $\nu_k$ for $\Omega_k$ satisfying $\langle \nu_k, \nabla d(x,\p M) \rangle \geq 0$.
\end{enumerate}
\end{lemma}

\begin{proof}[Proof of Proposition \ref{prop eigenfunction}]
(1). Direct calculation using Proposition \ref{thm laplace comparison} shows 
\begin{equation*}
    \Delta_p v = (p-1)|\vp'|^{p-2}\vp'' +|\vp'|^{p-2}\vp' \Delta d(x,\p M) \leq -\bar\l_p |v|^{p-2} v
\end{equation*}
on the set $M \setminus \text{Cut}(M)$. 
Since $v$ is smooth in any $\Omega_k$, we have for any nonnegative function $\eta \in C^1(\overline{M})$, 
\begin{eqnarray*}
&& \int_{\Omega_k} |\nabla v|^{p-2}\langle \nabla v, \nabla \eta \rangle \, d\mu_g \\
&=& -\int_{\Omega_k} \Delta_p v \, \eta \, d\mu_g +\int_{\p \Omega_k} |\nabla v|^{p-2} \frac{\p v}{\p \nu_k} \, \eta \, dA \\
&\geq & \bar\l_p \int_{\Omega_k} |v|^{p-2}v \eta \, d\mu_g +\int_{\p \Omega_k \cap \p M} |\nabla v|^{p-2} \frac{\p v}{\p \nu_k} \, \eta \, dA  +\int_{\p \Omega_k \setminus \p M} |\nabla v|^{p-2} \frac{\p v}{\p \nu_k} \, \eta \, dA \\
&=& \bar\l_p \int_{\Omega_k} |v|^{p-2}v \eta \, d\mu_g +\a \int_{\p M} |v|^{p-2}v \, \eta \, dA +\int_{\p \Omega_k \setminus \p M} |\nabla v|^{p-2} \vp' \langle \nabla d(x, \p M), \nu_k \rangle \, \eta \, dA\\
& \geq &  \bar\l_p \int_{\Omega_k} |v|^{p-2}v \eta\, d\mu_g  +\a \int_{\p M} |v|^{p-2}v \eta\, dA,
\end{eqnarray*}
where we used $\vp'>0$ and  $\langle \nu_k, \nabla d(x,\p M) \rangle \geq 0$.
Letting $k\to \infty$ yields that $v$ satisfies 
$$
\int_{M} |\nabla v|^{p-2}\langle \nabla v, \nabla \eta \rangle \, d\mu_g \ge \bar\l_p \int_{M} |v|^{p-1}v \eta \, d\mu_g +\a \int_{\p M} |v|^{p-2}v  \, \eta\, dA .
$$
Thus, we conclude that $v$ satisfies 
\begin{equation*}
    \begin{cases}
    -\Delta_{p} v  \geq \bar\l_p |v|^{p-2} v, & \text{ in } M, \\
    \frac{\p v}{\p \nu} |\nabla v|^{p-2} +\a |v|^{p-2}v \geq 0, & \text{ on } \p M,  
    \end{cases}
\end{equation*}
in the distributional sense.

(2).  The proof is similar to (1) and we omit the details. 
\end{proof}

\begin{proof}[Proof of Theorem \ref{Thm C}]
If $\a >0$, then by Proposition \ref{prop eigenfunction}, the function $v(x)=\vp(d(x,\p M)$ satisfies
\begin{equation*}
    \begin{cases}
    -\Delta_{p} v  \geq \bar\l_p |v|^{p-2} v, & \text{ in } M, \\
    \frac{\p v}{\p \nu} |\nabla v|^{p-2} +\a |v|^{p-2}v \geq 0, & \text{ on } \p M,  
    \end{cases}
\end{equation*}
in the distributional sense. 
Then Barta's inequality in Theorem \ref{Thm Barta} implies that $\l_p(M,\a) \geq \bar\l_p =\bar\l_p \left([0,R], \a\right)$.  The $\a <0$ case is completely similar. 

If the equality is in Theorem \ref{Thm C} achieved, then by the rigidity in part (1) of Theorem \ref{Thm Barta}, $v(x)=\vp(d(x, \p M)$ is indeed a constant mutilple of the first eigenfunction associated to $\l_p(M,\a)$. Same argument as in \cite[page 37]{Kasue84} or \cite[page 101]{Savo20} shows that $(M^n,g)$ is a $(\kappa, \Lambda)$-model space. 

\end{proof}

\section{Equality Case in Theorem \ref{Thm C} and Model Spaces}
In order to characterize the equality case in Theorem \ref{Thm C}, we need the notion of $(\kappa, \Lambda)$-model spaces introduced by Kasue \cite{Kasue84}. For this purpose, we introduce the following notations 
\begin{eqnarray*}
Z_{\kappa, \Lambda} &:=& \inf \{t>0 :  C_{\kappa, \Lambda}(t) =0\},\\
Y_{\kappa, \Lambda} &:=& \inf \{t \in (0, C_{\kappa, \Lambda}] :  C'_{\kappa, \Lambda}(t) =0\}.
\end{eqnarray*}
Here we understand $Z_{\kappa, \Lambda} =\infty$ if $C_{\kappa, \Lambda}$ does not vanish on $(0,\infty)$ and $Y_{\kappa, \Lambda} =\infty$ if $C'_{\kappa, \Lambda}$ does not vanish on $[0, C_{\kappa, \Lambda}]$. 
It's easy to see that $0< Z_{\kappa, \Lambda} <\infty$
if and only if either $\kappa >0$, or $\kappa=0$ and $\Lambda>0$, or $\kappa <0$, and that $\Lambda > \sqrt{|\kappa|}$ and $0< Y_{\kappa, \Lambda} <\infty$ if and only if either $\kappa >0$ and $\Lambda<0$, or $\kappa=0$ and $\Lambda =0$, or $\kappa <0$ and $0< \Lambda < \sqrt{|\kappa|}$.

Let $M^n(\kappa)$ denote the simply-connected $n$-dimensional space with constant sectional curvature $\kappa$. 
\begin{definition}\label{def model space}
A compact Riemannian manifold $(M^n,g)$ with boundary is called a $(\kappa, \Lambda)$-model space if one of the following conditions holds:
\begin{enumerate}
     \item $Z_{\kappa, \Lambda} < \infty$ and $M$ is isometric to the closed geodesic ball of radius $Z_{\kappa, \Lambda}$  in $M^n(\kappa)$.
     \item $\kappa=\Lambda=0$, or $0< Y_{\kappa, \Lambda} < \infty$. Moreover, $M$ is isometric to the warped product $[0, 2a] \times_{C_{\kappa, \Lambda}} \Gamma$, where $\Gamma$ is connected component of $\p M$ and $a$ is a positive number if $\kappa=\Lambda=0$, and $a=Y_{\kappa, \Lambda}$ if $0< Y_{\kappa, \Lambda} < \infty$. In this case, $\p M$ is disconnected. 
     \item $\kappa=\Lambda=0$, or $0< Y_{\kappa, \Lambda} < \infty$. Moreover, $\p M$ is connected and there is an involutive isometry $\sigma$ of $\p M$ without fixed points, and $M$ is isometric to the quotient space $[0, 2a] \times_{C_{\kappa, \Lambda}} \p M/G_{\sigma}$, where $a$ and $h$ are the same as in (2) and $G_{\sigma}$ is the isometry group on $[0, 2a] \times_{C_{\kappa, \Lambda}} \p M/G_{\sigma}$ whose elements consist of the identity and and the involutive isometry 
     $\hat{\sigma}$ defined by $\hat{\sigma}(t,x) =(2a-t, \sigma(x) )$. 
      \end{enumerate}
\end{definition}

From a standard argument (see for instance \cite[Section 1.3]{Kasue84} ), one sees that when $M$ is a $(\kappa, \Lambda)$-model space, the first Robin eigenfunction of the $p$-Laplacian can be written in the form 
\begin{equation*}
    u=\vp \circ d(x, \p M), 
\end{equation*}
where $\vp$ is a smooth function on $[0,R]$ satisfying 
\begin{equation*}
    \begin{cases} 
    (p-1)|\vp'|^{p-2} \vp'' +(n-1) T_{\kappa, \Lambda} |\vp'|^{p-2}\vp' =-\l_{p}(M,\a) |\vp|^{p-2}\vp, \\
    |\vp'(0)|^{p-2}\vp'(0)=\a |\vp(0)|^{p-2}\vp(0), \\
    \vp'(R)=0,
    \end{cases}
\end{equation*}
which gives the equality case in Theorem \ref{Thm C}. 

\section*{Acknowledgments}
The first author would like to thank Professor Richard Schoen for his support and interest in this work. 
Both authors are grateful to Professor Lei Ni for his encouragement and helpful conversations.

\bibliographystyle{plain}
\bibliography{ref}

\end{document}